\documentclass{amsart}
\usepackage{amsmath,amsxtra,amssymb,latexsym, amscd,amsthm}
\usepackage{graphicx,color}
\usepackage[utf8]{inputenc}
\usepackage[mathscr]{euscript}
\usepackage{mathrsfs}
\usepackage[english]{babel}
\usepackage{enumerate}
\usepackage{color}
\newtheorem {theorem}{Theorem}[section]
\newtheorem {corollary}{Corollary}[section]
\newtheorem {proposition}{Proposition}[section]
\newtheorem {lemma}{Lemma}[section]
\newtheorem {claim}{Claim}[section]

\theoremstyle{definition}
\newtheorem{definition}{Definition}[section]
\newtheorem {example}{Example}[section]
\newtheorem{remark}{Remark}[section]

\newcommand{\ord}{{\rm ord\ }}

\def\ar{a\kern-.370em\raise.16ex\hbox{\char95\kern-0.53ex\char'47}\kern.05em}
\def\ees{{\accent"5E e}\kern-.385em\raise.2ex\hbox{\char'23}\kern-.08em}
\def\eex{{\accent"5E e}\kern-.470em\raise.3ex\hbox{\char'176}}
\def\AR{A\kern-.46em\raise.80ex\hbox{\char95\kern-0.53ex\char'47}\kern.13em}
\def\EES{{\accent"5E E}\kern-.5em\raise.8ex\hbox{\char'23 }}
\def\EEX{{\accent"5E E}\kern-.60em\raise.9ex\hbox{\char'176}\kern.1em}
\def\ow{o\kern-.42em\raise.82ex\hbox{
  \vrule width .12em height .0ex depth .075ex \kern-0.16em \char'56}\kern-.07em}
\def\OW{O\kern-.460em\raise1.36ex\hbox{
\vrule width .13em height .0ex depth .075ex \kern-0.16em \char'56}\kern-.07em}
\def\uw{u\kern-.42em\raise.82ex\hbox{
  \vrule width .12em height .0ex depth .075ex \kern-0.16em \char'56}\kern-.07em}
\def\UW{U\kern-.42em\raise1.55ex\hbox{
\vrule width .13em height .0ex depth .075ex \kern-0.16em \char'56}\kern-.07em}
\def\DD{D\kern-.7em\raise0.4ex\hbox{\char '55}\kern.33em}
\def\OOH{{\accent"5E O}\kern-.78em\raise.8ex\hbox{\char'22}\kern.28em}
\def\UWS{\' \UW }
\title{Computation of the \L ojasiewicz exponents of real bivariate analytic functions}

\title{Computation of the \L ojasiewicz exponents of real bivariate analytic functions}
\author{S\~i Ti\d{\^e}p \DD inh$^1$}
\address{Institute of Mathematics, VAST, 18, Hoang Quoc Viet Road, Cau Giay District 10307, Hanoi, Vietnam}
\email{dstiep@math.ac.vn}

\author{Feng Guo$^2$}
\address{School of Mathematical Sciences, Dalian  University of Technology, Dalian, 116024, China}
\email{fguo@dlut.edu.cn}

\author{H\OOH NG \DD\UWS C NGUY\EEX N$^3$}
\address{TIMAS, Thang Long University, Nghiem Xuan Yem, Hanoi, Vietnam}
\address{Institute of Mathematics, VAST, 18, Hoang Quoc Viet Road, Cau Giay District 10307, Hanoi, Vietnam}
\email{duc.nh@thanglong.edu.vn}

\author{TI\EES N-S\OW N PH\d{A}M$^4$}
\address{Department of Mathematics, Dalat University, 1 Phu Dong Thien Vuong, Dalat, Vietnam}
\email{sonpt@dlu.edu.vn}

\thanks{$^{1, 3, 4}$These authors are funded by International Centre for Research and Postgraduate Training in Mathematics (ICRTM) under grant number ICRTM04$\_$2021.04}
\thanks{$^{2}$Feng Guo was supported by the Chinese National Natural Science Foundation under grant 11571350, the Fundamental Research Funds for the Central Universities.}

\date{\today}
\begin{document}
\maketitle
\begin{abstract}
The main goal of this paper is to present some explicit formulas for computing the {{\L}}ojasiewicz exponent in the {{\L}}ojasiewicz inequality comparing the rate of growth of two real bivariate analytic function germs. 
\end{abstract}
\section{Introduction}

The {\L}ojasiewicz inequalities and their variants play an important role in many branches of mathematics.
For example, {\L}ojasiewicz inequalities are very useful in the study of continuous regular functions, see \cite{Fichou2016, Kucharz2009} for pioneering works and \cite{Kucharz2019} for a survey. 
Also, {\L}ojasiewicz inequalities, together with Nullstellens\"atz, are crucial tools for the study of the ring of (bounded) continuous semi-algebraic functions on a semi-algebraic set, see \cite{Fernando2014-1, Fernando2014-2, Fernando2015}. 

Let $f,g \colon (\mathbb{R}^n, 0) \to (\mathbb{R}, 0)$ be nonzero real analytic function germs. 
Assume that $0\in\{f=0\}\subset \{g=0\}$. By the classical \L ojasiewicz inequality on comparing the rate of growth, there exist positive constants $C,r$ and $\alpha$ such that  
\begin{equation}\label{Lojasiewicz}
|f(x)|\geqslant C |g(x)|^\alpha\ \text{ for }\ |x|\leqslant r. 
\end{equation}
The infimum of such $\alpha$ is called the {\em \L ojasiewicz exponent of $f$ w.r.t. $g$} and denoted by $\mathscr{L}_g(f)$.  

Note that several versions of the \L ojasiewicz inequality have been studied for a special case where $g$ is the distance function to the zero set of $f,$ see \cite{Dinh2017-2, Dinh2012, Dinh2014-1, Dinh2013, Dinh2019, Dinh2016-1, HaHV2013, HaHV2015-1,Kuo74}.
Furthermore, the computation or estimation of \L ojasiewicz exponents in this case has been considered in these works. In \cite{Dinh2021-2}, the authors provided a global version of the {{\L}}ojasiewicz inequality on comparing the rate of growth of two polynomial functions in the case the mapping defined by these functions is (Newton) non-degenerate at infinity. However, no computation or estimation of \L ojasiewicz exponents has been given.

 In this work, we will address partially to this problem by giving some explicit formulas for computing the \L ojasiewicz exponent $\mathscr{L}_g(f)$ in the most general case when $f$ and $g$ are two arbitrary real bivariate analytic function germs. Moreover, our proof provides a new algorithm computing the limit of bivariate rational functions (See Corollary \ref{limit}).

The rest of the paper is organized as follows. 
In Section~\ref{Section2}, we recall the notions of Newton polygon relative to an arc and sliding due to Kuo and Parusi\'nski which are crucial in the proof of our  formulas for the \L ojasiewicz exponent, which are our main results (Theorem~\ref{thm31} and Theorem~\ref{thm32}), whose statements, together with the proofs, will be given in Section~\ref{Section3}.

\section{The Newton polygon relative to an arc} \label{Section2}

The technique of Newton polygons plays an important role in this paper. It is well-known that Newton transformations which arise in a natural way when applying the Newton algorithm provide a useful tool for calculating invariants of singularities. For a complete treatment we refer to \cite{ Brieskorn1986, Casas-Alvero2000, Walker1950, Wall2004}. In this section we recall the notion of Newton polygon relative to an arc due to Kuo and Parusi\'nski \cite{Kuo2000} (see also, \cite{HaHV2008-6} and \cite{HaHV2010-1}).

Let $\mathbb K:=\mathbb R$ or $\mathbb K:=\mathbb C$ and let $f \colon (\mathbb{K}^2, 0) \to (\mathbb{K}, 0)$ denote a nonzero analytic function germ with Taylor expansion:
$$f(x, y) = f_m(x, y) + f_{m + 1}(x, y) +  \cdots,$$
where $f_k$ is a homogeneous polynomial of degree $k,$ and $f_m \not \equiv 0.$ For the remainder of the paper, we will assume that $f$ is {\em regular in $x$ of order $m$} in the sense that $f_m(1, 0) \ne 0.$
(This can be achieved by a linear transformation $x' = x, y' = y + cx,$ where $c$ is a generic number). 
Let $\phi$ be an analytic arc in $\mathbb{K}^2$, which is not tangent to the $x$-axis. Then it can be parametrized by 
$$x=c_1t^{n_1} + c_2t^{n_2}+ \cdots \in \mathbb{K}\{t\} \text{ and } y=t^N$$
and therefore can be identified with a {\em Puiseux series} (denoted also by $\phi$ for simplicity of notation)
\begin{eqnarray*}
x = \phi(y) = c_1y^{n_1/N} + c_2y^{n_2/N}+ \cdots \in \mathbb{K}\{y^{1/N}\}
\end{eqnarray*}
with $N \le n_1 < n_2 < \cdots $ being positive integers. 
The changes of variables $X := x - \phi(y)$ and $Y := y$ yield
$$F(X, Y) := f(X + \phi(Y), Y) := \sum c_{ij}X^iY^{j/N}.$$
For each  $c_{ij} \ne 0,$ let us plot a dot at $(i, j/N),$ called a {\em Newton dot.} The set of Newton dots is called the {\em Newton diagram.} 
They generate a convex hull, whose boundary is called the {\em Newton polygon of $f$ relative to $\phi,$}  to be denoted by $\mathbb{P}(f, \phi).$ Note that this is the Newton polygon of $F$ in the usual sense. If $\phi$ is a {\em Newton--Puiseux root} of $f $ (i.e., $f(\phi(y), y)=0$), then there are no Newton dots on $X = 0$, and vice versa. Assume that $\phi$ is not a Newton--Puiseux root of $f$, then the exponents of the series $f(\phi(y), y) = F(0, Y)$ correspond to the Newton dots on the line $X = 0.$ In particular, $\mathrm{ord} f(\phi(y), y) = h_0,$ where $(0, h_0)$ is the lowest Newton dot on $X = 0$. 

The {\em highest Newton edge}, denoted by $E_H$ (or $E_1$) is defined as follows: If $\phi$ is a Newton--Puiseux root of $f$, then $E_1$ is the non-compact edge of the polygon $\mathbb{P}(f, \phi)$ parallel to the $y$-axis. If $\phi$ is not a Newton--Puiseux root of $f$, then $E_1$ is the compact edge of  the polygon $\mathbb{P}(f, \phi)$ with a vertex being the lowest Newton dot on $X = 0.$ The {\em Newton edges} $E_2, E_3,\dots,E_s$ are compact adges of $\mathbb{P}(f, \phi)$. These edges and their associated {\em Newton angles} $\theta_2,\dots,\theta_s$ are defined in an obvious way as illustrated in the following example.
\begin{example}\label{Example1}{\rm
Take $f(x, y) := {x}^{3}-{y}^{5}+{y}^{6}$ and $\phi(y): = y^{5/3}.$ 
We have 
$$F(X, Y) := f(X + \phi(Y), Y) = {X}^{3}+3\,{X}^{2}{Y}^{5/3}+3\,X{Y}^{10/3}+{Y}^{6}.$$ 
By definition, the Newton polygon of $f$ relative to $\phi$ has three edges $E_1, E_2$ with $\tan \theta_1 = 8/3$ and $\tan \theta_2 = 5/3$ (see Figure~\ref{Figure1}).

\unitlength = .7cm
\begin{figure}
\begin{picture}(7, 8)(-2, -1)
\put (0, 0){\vector(0, 1){7}}
\put (0, 0){\vector(1, 0){7}}

\multiput(0, 1) (.5, 0){14}{\line(1, 0){.1}}
\multiput(0, 2) (.5, 0){14}{\line(1, 0){.1}}
\multiput(0, 3) (.5, 0){14}{\line(1, 0){.1}}
\multiput(0, 4) (.5, 0){14}{\line(1, 0){.1}}
\multiput(0, 5) (.5, 0){14}{\line(1, 0){.1}}
\multiput(0, 6) (.5, 0){14}{\line(1, 0){.1}}

\multiput(1, 0) (0, .5){14}{\line(0, 1){.1}}
\multiput(2, 0) (0, .5){14}{\line(0, 1){.1}}
\multiput(3, 0) (0, .5){14}{\line(0, 1){.1}}
\multiput(4, 0) (0, .5){14}{\line(0, 1){.1}}
\multiput(5, 0) (0, .5){14}{\line(0, 1){.1}}
\multiput(6, 0) (0, .5){14}{\line(0, 1){.1}}

\multiput(1, 3.333) (-.45, 0){2}{\line(-1, 0){.2}}
\multiput(2, 1.666) (-0.45, 0){4}{\line(-1, 0){.2}}

\put(3, 0){\circle*{0.2}}
\put(2, 1.666){\circle*{0.2}}
\put(1, 3.333){\circle*{0.2}}
\put(0, 6){\circle*{0.2}}

\put(1.8, 2.4){$E_2$}
\put(.8, 4.2){$E_1=E_H$}

\put(2.2, .15){$\theta_2$}
\put(0.3, 3.5){$\theta_1$}

\thicklines
\put(3, 0){\line(-3, 5){2}}
\put(1, 3.333){\line(-2, 5){1.05}}

\put(-.75, 5.85){$\ 6$}
\put(-.75, 3.1){$\frac{10}{3}$}
\put(-.75, 1.5){$\frac{5}{3}$}
\put(.85, -1){$1$}
\put(1.85, -1){$2$}
\put(2.85, -1){$3$}

\put(1, -.15){\line(0, 1){.3}}
\put(2, -.15){\line(0, 1){.3}}
\put(-.15, 3.333){\line(1, 0){.3}}
\put(-.15, 1.666){\line(1, 0){.3}}
\end{picture}
\caption{ \ } \label{Figure1}
\end{figure}
}\end{example}

Take any edge $E_s.$ The {\em associated polynomial} $\mathcal{E}_s(z)$ is defined to be $\mathcal{E}_s(z) := \mathcal{E}_s(z, 1),$ where
$$\mathcal{E}_s(X, Y) := \sum_{(i, j/N) \in E_s} c_{ij} X^i Y^{j/N}.$$ 
Next, let us recall the notion of {\em sliding} (see \cite{Kuo2000}). Suppose that $\phi$ is not a Newton--Puiseux root of $f.$  Consider the Newton polygon $\mathbb{P}(f, \phi).$ Take any nonzero root $c$ of $\mathcal{E}_1(z) = 0,$ the polynomial equation associated to the highest Newton edge $E_1.$ We call
$$\phi_1\colon x = \phi(y) + cy^{\tan \theta_1}$$
a {\em sliding} of $\phi$ along $f,$ where $\theta_1$ is the angle associated to $E_1.$ A recursive sliding 
$$\phi \to \phi_1 \to \phi_2 \to \cdots$$ produces a limit, denoted by $\phi_\infty$, which is a Newton--Puiseux root of $f$. 
The series $\phi_\infty$ will be called a {\em final result of sliding $\phi$ along $f$}. Note that $\phi_\infty$ has the form
$$\phi_\infty\colon x = \phi(y) + c y^{\tan \theta_1}+\text{higher order terms},$$
due to the following technical lemma. 

\begin{lemma} \label{Lemma22}
Let $\phi$ be a Puiseux series, which is not a Newton--Puiseux root of $f$. 
Let $\theta_1$ and $\mathcal{E}_1$ be respectively the Newton angle and polynomial associated to the highest Newton edge $E_1.$ Consider a series of the following form
$$\psi\colon x = \phi(y) + c y^{\rho} + \textrm{ higher order terms,}$$
where $c \in \mathbb{K}$ and $\rho \in \mathbb{Q}, \rho > 0.$ Then the following statements hold:
\begin{enumerate}[{\rm (i)}]
\item If either $c$ or $\rho$ is generic (i.e., $\arctan \rho$ is not a Newton angle of $\mathbb{P}(f, \phi)$; or $\arctan \rho$ is a Newton angle of $\mathbb{P}(f, \phi)$ but $c$ is not a root of the polynomial associated to the Newton edge with Newton angle $\arctan \rho$), then  
 $$\mathrm{ord} f(\psi(y), y)=\min \{a\rho+b\ |\ (a,b)\in \mathbb P(f,\phi)\}.$$
Furthermore, 
$$\mathrm{ord} f(\psi(y), y)\leqslant\mathrm{ord} f(\phi(y), y).$$
In particular, if either $\tan \theta_{H}< \rho$ or $\tan \theta_{H}= \rho$ and $\mathcal{E}_H(c) \ne 0$ then 
$\mathbb{P}(f, \psi) = \mathbb{P}(f, \phi),$ and therefore 
$$\mathrm{ord} f(\psi(y), y) = \mathrm{ord} f(\phi(y), y).$$

\item If $\tan \theta_{H}= \rho$ and $\mathcal{E}_H(c)  = 0$ then 
$$\mathrm{ord} f(\psi(y), y) > \mathrm{ord} f(\phi(y), y).$$
\end{enumerate}
\end{lemma}

\begin{proof} cf. \cite{Brieskorn1986, Casas-Alvero2000, Walker1950}. For a detailed proof, we refer to \cite{HaHV2008-6}. In fact, the special case where $\psi(y) = \phi(y) + c y^{\tan \theta_1}$ was proved in \cite[Lemma 2.1]{HaHV2008-6}. Then the lemma is deduced by applying the special case (possibly infinitely) many times.
\end{proof}
\begin{definition}\label{def21}{\rm
 For each Puiseux series $\phi(y) = \sum_{i} a_i y^{\alpha_i}$ and for each positive real number $\rho$,  the {\em $\rho$-approximation} of $\phi(y)$ is defined to be the series $\sum_{\alpha_i < \rho}a_{i}y^{\alpha_i} + cy^\rho,$ where $c$ is a generic real number. %For two distinct Puiseux series $\phi_1(y)$ and $\phi_2(y)$, $\rho := \mathrm{ord}\ (\phi_1(y) - \phi_2(y))$ is the {\em contact order} of $\phi_1(y)$ and $\phi_2(y)$. And the {\em approximation of $\phi_1(y)$ and $\phi_2(y)$} is defined to be the $\rho$-approximation series of $\phi_1(y)$ (and hence of $\phi_2(y)$). 
We associate to any Puiseux series $\phi$ its {\em real approximation} $\phi^{\mathbb R}(y)$ defined to be the $\rho$-approximation of $\phi$, where $\rho$ is the smallest exponent occurring in $\phi$ with non-real coefficient. 
It is clear that if $\varphi$ is {\em real}, i.e., all coefficients of $\varphi$ are real, then the real approximation of $\varphi$ is itself.
Now, for $f\in\mathbb K\{x,y\}$ which is regular in $x$, let $\mathcal V_\mathbb R(f)$ be the set of all real approximations of non-real Newton--Puiseux roots of $f$. 

For any two distinct series $\phi_1,\phi_2$, their {\em approximation}, denoted by $\phi_{1,2}$, is defined to be the $\rho$-approximation of $\phi_1$ where $\rho:=\mathrm{ord}(\phi_1-\phi_2)$. 
Let $\mathcal V_a(f)$ be the set of all approximations of $\phi_1$ and $\phi_2$ with $\phi_1\neq \phi_2$ being Newton--Puiseux roots of $f$. Note that $\mathcal V_\mathbb R(f)\subset \mathcal V_a(f)$ if $f\in \mathbb R\{x,y\}$.

}\end{definition}
The following useful assertion is a direct consequence of Lemma \ref{Lemma22}:
\begin{lemma} \label{Lemma23}
Assume that $\mathbb K=\mathbb R$. Let $\phi$ be a Puiseux series and let ${E}_1,\ldots,{E}_s$ be the Newton edges of $\mathbb P(f,\phi)$. Let $\theta_i$ and $\mathcal{E}_i$ be the corresponding Newton angle and polynomial associated to $E_i.$ Then by a permutation of indexes, we have
$$\pi/2\geqslant\theta_1>\theta_2>\ldots>\theta_s$$
and the following statements hold: 
\begin{enumerate}[{\rm (i)}]
\item If $\mathcal{E}_i$ has two distinct roots, there exists $\psi\in \mathcal V_a(f)$ being of the form
	$$\psi(y) = \phi(y) + c y^{\tan\theta_i} + \textrm{	higher order terms},$$
where $c$ is a generic number.
\item If $\theta\ne\theta_1$ is a Newton angle, then there exists $\psi\in \mathcal V_a(f)$ being of the form
$$\psi(y) = \phi(y) + c y^{\tan\theta} + \textrm{ higher order terms},$$
where $c$ is a generic number.
\item If $\mathcal{E}_i$ has a non-real root, then there exists $\psi\in \mathcal V_\mathbb R(f)$ being of the form
$$\psi(y) = \phi(y) + c y^{\tan\theta_i}  + \textrm{ higher order terms},$$
where $c$ is a generic number.
\end{enumerate}
\end{lemma}

\section{Formulas for \L ojasiewicz exponents}\label{Section3}
For the remainder of this section, let $f,g \colon (\mathbb{R}^2, 0) \to (\mathbb{R}, 0)$ be nonzero real analytic function germs, which are regular in $x$ such that $0\in\{f=0\}\subset \{g=0\}$. By the classical \L ojasiewicz inequality (see, for instance \cite{Lojasiewicz1991}), there exists positive constants $C,r$ and $\alpha$ such that  
$$|f(x,y)|\geqslant C |g(x,y)|^\alpha\ \text{ for }\ |x|\leqslant r.$$ 
The infimum of such $\alpha$ is called the {\em \L ojasiewicz exponent of $f$ with respect to $g$} and denoted by $\mathscr{L}_g(f)$.  

Take any analytic arc $\phi\in \mathbb{R}^2$ at the origin parametrized by $\left(x(t),y(t)\right)$.  
If $g \circ \phi \not \equiv 0,$ then we can define the following positive rational number %$\ell(\phi)$
\begin{eqnarray*}
\ell(\phi)&:=&\frac{\mathrm{ord} f(\phi(t))}{\mathrm{ord} g(\phi(t))}.
\end{eqnarray*}
By the Curve Selection Lemma (see \cite[Lemma 3.1]{Milnor1968}), it is not hard to show that % the {\L}ojasiewicz exponent of of $g$ w.r.t. $f$ is given by
\begin{eqnarray}\label{Eqn1}
\mathscr{L}_g(f) &=& \sup_\phi \ell(\phi),
\end{eqnarray}
where the supremum is taken over all analytic arcs passing through the origin, which are not contained in the zero locus of $g.$ 
%{(The exponent $\alpha$ depends on $r$ so this is the exponent when $r\to 0^+$).}
Furthermore, since $f$ and $g$ are $x$-regular, the supremum in~\eqref{Eqn1} can be taken over all real analytic arcs passing through the origin not contained in the zero locus of $y.$ 
\begin{remark}
Note that the supremum in~\eqref{Eqn1} may not be attained, i.e., it is possible that there is no analytic arc $\phi\in \mathbb{R}^2$ at the origin such that $\mathscr{L}_g(f)=\ell(\phi).$
The following example is an illustration. 

Let 
$$f(x,y)=x^2\ \quad \text{ and }\quad g(x,y)=x(x^2+y^2).$$
Then $f(x,y)\geqslant g^2(x,y)$  for $(x,y)$ closed enough to the origin.
So $\mathscr L_g(f)\leqslant 2.$
On the other hand, for each positive integer $k$, let $\phi_k(t)=(t,t^{1/k})$, then we have
$$\ell(\phi_k)=\frac{2}{1+\frac{2}{k}}\to 2 \quad \ \text{ as }\ k\to+\infty.$$
Therefore $\mathscr L_g(f)= 2.$  Now, for any analytic arc $\phi(t)=(x(t),y(t))$ at the origin, we have 
$$\ell(\phi)=\dfrac{2 \ord x(t)}{\ord x(t)+2\min\{\ord x(t),\ord y(t)\}}<\dfrac{2 \ord x(t)}{\ord x(t)}=2,$$
i.e., $\mathscr L_g(f)$ is not attained for any analytic arc $\phi\in \mathbb{R}^2$ at the origin.
\end{remark}

\subsection{First formula for the \L ojasiewicz exponents} 
%Let $f,g \colon (\mathbb{R}^2, 0) \to (\mathbb{R}, 0)$ be {nonzero} real analytic function germ, which are regular in $x$. 
Let $\beta_j,\ j=1,\ldots,k$ be the common real Newton--Puiseux roots of $f$ and $g$ of multiplicities $m_j$ and $n_j$ respectively. Let $\mathcal V_a(fg)$ be the set given by Definition \ref{def21}. 
%It is well-known that, by the Curve Selection Lemma, we can show that $$\mathcal{L}_g(f)=\sup\{\ell(\phi) | \mathrm{ord} g(\phi(y),y)\},$$ where the supremum is taken over all Puiseux series $\phi$, and $\ell(\phi):=\dfrac{\mathrm{ord} g((\phi(y),y)}{\mathrm{ord} g((\phi(y),y)}$.
For any $\phi\in \mathcal V_a(fg)$, we will write $\ell(\phi)$ instead
of $\ell((\phi(y),y))$ for simplicity.
\begin{theorem}\label{thm31}
Define
$$\mathcal{L}^+_g(f):=\max\left\{\ell(\phi),\frac{m_j}{n_j}\ \Big|\ \phi\in \mathcal V_a(fg),\  j=1,\ldots,k \right\}\ \text{ and }\ \mathcal{L}^-_g(f)=\mathcal{L}^+_{\bar g}(\bar f),$$
where $\bar f(x,y):=f(x,-y)$ and $\bar g(x,y):=g(x,-y)$.
Then the \L ojasiewicz  exponent of $g$ w.r.t $f$ is given by
$$\mathscr{L}_g(f) = \max\left\{\mathcal{L}^+_g(f), \mathcal{L}^-_g(f) \right\}.$$
\end{theorem}

\begin{proof}
We first show that
\begin{equation}\label{one-side}
\mathscr{L}_g(f) \geqslant  \max\left\{\mathcal{L}^+_g(f), \mathcal{L}^-_g(f)\right\}.
\end{equation}
By (\ref{Eqn1}), it is obvious that 
$$\mathscr{L}_g(f) \geqslant \max\left\{\ell(\phi)\ |\ \phi\in\mathcal V_a(fg)\right\}.$$ 
Therefore  we only  need to show that 
\begin{equation}\label{>=mul}
\mathscr{L}_g(f) \geqslant \frac{m_j}{n_j}\ \text{ for all }\ j=1,\ldots,k.
\end{equation}
To do this, fix $j\in\{1,\dots,k\}$ and consider the Newton polygons $\mathbb P(f,\beta_j)$ of $f$ and $\mathbb P(g,\beta_j)$ of $g$ relative to the arc $\beta_j$. 
Let $A_1=(x_1,y_1)$ and $A_2=(x_2,y_2)$ be respectively the vertices of $\mathbb P(f,\beta_j)$ and $\mathbb P(g,\beta_j)$ being closest to the $y$-axis. 
We will show that
\begin{equation*}\label{multiplicity}
x_1=m_j\ \text{ and }\ x_2=n_j.
\end{equation*}
Indeed, %let us prove the first equality, the second one can be proved similarly.
in view of Puiseux's theorem (see, for example \cite[page~98]{Walker1950}), we can write
\begin{equation}\label{Puiseux-f}
f(x,y)=(x-\beta_j(y))^{m_j}h(x,y),
\end{equation} 
where $h(\beta_j(y),y)\neq 0$ for all $j$. So 
$$f(X + \beta_j(Y), Y)=X^{m_j}h(X + \beta_j(Y)Y).$$
This implies $x_1=m_j$ and $y_1=\ord h\left( \beta_j(y),y)\right)$ and similarly $ x_2=n_j$. 

For each positive integer $n$, define a new arc 
$$x=\phi_n(y)=\beta_j(y)+y^n.$$ 
By~\eqref{Puiseux-f}, we have
$$ f(\phi_n(y),y)=y^{n x_1}h(\phi_n(y),y) .$$
For $n$ large enough, we have 
$$y_1=\ord h\left( \beta_j(y),y)\right) =\ord h\left( \phi_n(y),y)\right).$$
So this yields 
$\ord{f(\phi_n(y),y)}=n x_1+y_1.$
By the same way, we also have $\ord{g(\phi_n(y),y)}=n x_2+y_2.$
Consequently,
$$\ell (\phi_n)=\frac{nx_1+y_1}{nx_2+y_2}.$$
Note that 
$$\lim\limits_{n\to \infty} \ell (\phi_n)=\dfrac{x_1}{x_2}=\dfrac{m_j}{n_j},$$ 
so~\eqref{>=mul} follows from~\eqref{Eqn1}.
%\begin{align}
%\mathscr{L}_g(f)\geqslant \dfrac{m_j}{n_j}.
%\end{align}
Therefore, $\mathscr{L}_g(f) \geqslant  \mathcal{L}^+_g(f)$. 
Similarly, one has $\mathscr{L}_g(f) \geqslant  \mathcal{L}^-_g(f)$ and hence the inequality~\eqref{one-side} holds.
Now we need to show that the inequality in~\eqref{one-side} is actually an equality.
%$$\mathscr{L}_g(f) \geqslant  \max\left\{\mathcal{L}^+_g(f), \mathcal{L}^-_g(f)\right\}.$$
%\end{proof}
%\begin{lemma}
%If the \L ojasiewicz exponent $\mathscr{L}_g(f)$ is attained at a given curve, then
%$$\mathscr{L}_g(f) = \max\left\{\mathcal{L}_+(g), \mathcal{L}_-(g)\right\}.$$
%\end{lemma}
%\begin{proof}
%Without loss of generality we may assume that $\mathscr{L}_g(f) = \mathcal{L}_+(g)$. Assmue 

Suppose for contradiction that
 $$\mathscr{L}_g(f) >\max\left\{\mathcal{L}^+_g(f), \mathcal{L}^-_g(f)\right\}.$$
Then there is a real analytic arc $\phi$ passing through the origin and not lying in the $x$-axis such that $g\circ \phi\not\equiv 0$ and  
$$ \ell(\phi)>\max\left\{\mathcal{L}^+_g(f), \mathcal{L}^-_g(f)\right\}.$$
%{ Since $f$ is nonzero and regular in $x$, it follows that such a curve $\phi$ does not lie in the $x$-axis and so its germ at the origin can be parametrized by either $(x = x(t), y = t)$ or $(x = x(t), y = -t),$ where $x(t)$ is an element in $\mathbb{R}\{t^{1/N}\}$ for some positive integer $N$ with $x(0)=0$.}
Note that $\phi$ can be parametrized by either 
$$(x = \phi(t), y = t)\quad \text{ or }\quad (x = \phi(t), y = -t),$$ 
where $\phi(t)$ is an element in $\mathbb{R}\{t^{1/N}\}$ for some positive integer $N$ with $\phi(0)=0$. 
%We denote by $\Gamma_+$ and $ \Gamma_-$ the set of all curves in $\Gamma$ being of form $(x=x(t),y= t)$ and $(x = x(t), y = -t)$, respectively. 
%Without loss of generality we may assume that $\Gamma_+ \neq \emptyset$. Pick $\phi\in\Gamma_+$.
%Let $\mathbb P_1$ and $\mathbb P_2$ be the Newton polygons of $f$ and $g$ relative to $\phi$ respectively. 
Without loss of generality we may assume that $\phi$ can be parametrized by $(x = \phi(t), y = t)$.
Denote by $E_1$ and $E_2$ the highest Newton edges of $\mathbb P(f,\phi)$ and $\mathbb P(g,\phi)$ respectively.
Let $\mathcal{E}_i$ and $\theta_i$ be respectively its associated polynomial and Newton angle. %We first see that, $\tan\theta_1\geqslant\tan \theta_2$. 

\begin{claim}\label{Claim21} 
We have $\tan \theta_1=\tan \theta_2$.
\end{claim}
\begin{proof}
%Arguing by contradiction we suppose that $c\in \mathbb{R}$ is a root of $\mathcal{E}_1(z)$.
Assume for contradiction that, $\tan \theta_1>\tan \theta_2$. Let $\phi_\infty$ be a final result of sliding $\phi$ along $f$. Write 
$$\phi_\infty(y)=\phi(y)+\sum_{i\geqslant 1}a_iy^{\alpha_i},$$
where $a_i\in\mathbb C\setminus\{0\}$, $\tan \theta_1=\alpha_1<\alpha_2<\cdots$. 
We will show that $a_i\in \mathbb R$ for all $ i\geqslant 1$.
In fact, if this is not the case, for each $n\geqslant 0$, define the series 
$$\phi_0(y):=\phi(y),\ \ \ \phi_n(y):=\phi(y)+\sum^{n}_{i= 1}a_i
y^{\alpha_i}\ \text{ for }\ n\geqslant 1,$$
and let $n_0$ be the smallest index such that $a_{n_0}\not\in \mathbb R$. 
Then $n_{0}\geqslant 1$ and 
$$\phi^\mathbb R_{n_0}(y)=\phi_{n_0-1}(y)+c y^{\alpha_{n_0}}
+ \textrm{ higher order terms} ,$$
where $c\in \mathbb R$ is a generic number.  
By applying Lemma \ref{Lemma22}, we obtain 
$$\ord f(\phi^\mathbb R_{n_0}(y),y)=\ord f(\phi_{n_0-1}(y),y)>\cdots>\ord f(\phi(y),y)$$
and
$$\ord g(\phi^\mathbb R_{n_0}(y),y)=\ord g(\phi_{n_0-1}(y),y)=\cdots=\ord g(\phi(y),y).$$
So
$$\ell(\phi^\mathbb R_{n_0})= \ell(\phi_{n_0-1})>\ldots>\ell(\phi)>\mathcal{L}^+_g(f),$$
a contradiction, since $\phi^\mathbb R_{n_0}\in \mathcal V_{\mathbb R}(f)\subset \mathcal V_{\mathbb R}(fg)\subset \mathcal V_a(fg)$. This shows that $a_n\in\mathbb R$ for all $n\geqslant 1$. But then this contradicts to the assumption that  $\{f=0\}\subset \{g=0\}$ in $\mathbb R^2$, hence 
$$\tan \theta_1\leqslant\tan \theta_2.$$

Now assume for contradiction that $\tan \theta_1 <\tan \theta_2$. 
Note that, $\theta_1$ and $\theta_2$ are Newton angles of $\mathbb P(fg,\phi)$. %{ (Need a reference: Minkowski sum of a product, Brieskorn?)}
Then by Lemma \ref{Lemma23}(ii), there exists $\psi\in \mathcal V_a(fg)$ being of the form
$$\psi(y)=\phi(y)+c y^{\tan\theta_1}+ \textrm{ higher order terms},$$
where $c\in \mathbb R$ is a generic number. It follows from Lemma \ref{Lemma22}(i) that
$$\ord f\left(\psi(y),y\right) = \ord f\left(\phi(y),y\right)\quad \text{ and }\quad \ord g\left(\psi(y),y\right)\leqslant\ord g\left(\phi(y),y\right),$$
and hence $\ell(\psi)\geqslant\ell(\phi)>\mathcal{L}^+_g(f)$.
This contradiction finishes the claim.
\end{proof}
\begin{claim}\label{Claim23} 
The polynomial $\mathcal{E}_1 \mathcal{E}_2$ has only one root. 
\end{claim}
\begin{proof}
Assume for contradiction that $\mathcal{E}_1 \mathcal{E}_2$ has two distinct roots. 
By Claim \ref{Claim21}, $\theta_1=\theta_2$, so $\mathcal{E}_1 \mathcal{E}_2$ is the Newton polynomial associated to the highest Newton edge of $\mathbb P(f g,\phi)$ (with Newton angle $\theta_1$). 
Then by Lemma \ref{Lemma23}(i), there exists $\psi\in \mathcal V_a(fg)$ being of the form
$$\psi(y)=\phi(y)+c y^{\tan\theta_1}+\textrm{ higher order terms},$$
where $c\in \mathbb R$ is a generic number. 
Then Lemma \ref{Lemma22}(ii) yields 
$$\ord f\left(\psi(y),y\right)=\ord f\left(\phi(y),y\right)\text{ and }\ord g\left(\psi(y),y\right)=\ord g\left(\phi(y),y\right).$$
Hence $\ell(\phi)= \ell(\psi)$ and so $\ell(\phi)\leqslant\mathcal{L}^+_g(f)$ which contradicts the assumption $\ell(\phi)>\mathcal{L}^+_g(f)$.
\end{proof}
Let $a\in\mathbb R$ be the unique root of the polynomial $\mathcal{E}_1 \mathcal{E}_2$ and let $\widetilde\phi(y):=\phi(y)+ay^{\tan\theta_1}$. We denote by $\widetilde E_1$ and $\widetilde E_2$ the highest Newton edge of $\mathbb P(f,\widetilde\phi)$ and $\mathbb P(g,\widetilde\phi)$ respectively. 
For $i=1,2,$ let $\widetilde\theta_i$ and $\widetilde{\mathcal{E}}_i$ be the Newton angle and the polynomial associated to $\widetilde E_i .$ 
Recall that $E_1$ and $E_2$ are respectively the highest Newton edges of $\mathbb P(f,\phi)$ and $\mathbb P(g,\phi)$.
Let $B_i=(x_i,y_i)$ be the vertex of {$E_i$} which is not contained in the $y$-axis. 
%Note that $x_i=\deg \mathcal{E}_i$.

\begin{claim}\label{Claim24}  If $\widetilde\phi(y)$ is not a
	Newton--Puiseux root of $f$, then the following properties hold:
\begin{enumerate}[{\rm (i)}]
\item $B_i$ is a vertex of $\widetilde E_i$, therefore  $\deg \widetilde{\mathcal{E}}_i=x_i=\deg{\mathcal{E}}_i$.
\item $\tan \widetilde\theta_1=\tan \widetilde\theta_2$.
\item The polynomial $\widetilde{\mathcal{E}}_1 \widetilde{\mathcal{E}}_2$ has only one root. 
\item $\ell(\widetilde\phi)\geqslant\ell(\phi)$. %{(Is it necessary to prove this?)} %>\mathcal{L}^+_g(f)$.%$\widetilde\phi\in \Gamma_+$.
\end{enumerate}
\end{claim}

\begin{proof}
(i) Let us define the function 
$$\mu(t)=\dfrac{tx_1+y_1}{tx_2+y_2}.$$
We first claim that $x_1y_2\geqslant x_2y_1$. In fact, if this is not the case, i.e., $x_1y_2< x_2y_1$, then the function $\mu$ is strictly decreasing and $y_1>0$.
 Since $f$ is regular in $x$, there exists a Newton edge $E$ of $\mathbb P(f,\phi)$ which is different from $E_1$ and has $B_1$ as a vertex. Let $\theta$ be the Newton angle associated to $E$. 
Clearly $\theta<\theta_1=\theta_2$, therefore, by Lemma
\ref{Lemma23}(ii), there exists $\psi\in\mathcal V_a(f)\subset \mathcal V_a(fg)$ such that
$$\psi(y)=\phi(y)+c y^{\tan\theta}+\textrm{ higher order terms},$$
where $c\in \mathbb R$ is a generic number. 
We have 
$$\ord f(\phi(y),y)= x_1\tan\theta_1+y_1\quad \text{ and }\quad \ord g\left(\phi(y),y\right)= x_2\tan\theta_2+y_2=x_2\tan\theta_1+y_2,$$
Moreover, by Lemma \ref{Lemma22}(i) and the choice of the edge $E$,
$$\ord f(\psi(y),y)= x_1\tan\theta+y_1\quad \text{ and }\quad \ord g\left(\psi(y),y\right)\leqslant x_2\tan\theta+y_2.$$
Hence 
$$\ell(\phi)=\mu \tan\theta_1<\mu \tan\theta\leqslant\ell(\psi)\leqslant \mathcal{L}^+_g(f),$$ 
which is a contradiction. 
Hence we must have $x_1y_2\geqslant x_2y_1$, i.e., the function $\mu$ is increasing. 

Let $\widetilde E$ be the edge of $\mathbb P(f,\widetilde\phi)$ such
that $B_1$ is the vertex having larger $x$-coordinate. Let $\widetilde
\theta$ be the Newton angle associated to $\widetilde E$. Since
$\widetilde\phi(y)$ is not a Newton--Puiseux root of $f$, $\widetilde E$ is a compact edge and therefore $\theta_1<\widetilde \theta<\pi/2$. If $\widetilde E$ is not the highest Newton edge of $\mathbb P(f,\widetilde\phi)$, then by Lemma \ref{Lemma23}(ii), there exists $\varphi\in\mathcal V_a(f)\subset \mathcal V_a(fg)$ such that
$$\varphi(y)=\widetilde\phi(y)+c y^{\tan\widetilde\theta}+\textrm{ higher order terms},$$
where $c\in \mathbb R$ is a generic number. It follows from Lemma \ref{Lemma22}(i) that
$$\ord f(\varphi(y),y)= x_1\tan\widetilde\theta+y_1\quad \text{ and }\quad\ord g\left(\varphi(y),y\right)\leqslant x_2\tan\widetilde\theta+y_2.$$
Hence 
$$\ell(\phi)=\mu(\tan\theta_1)\leqslant \mu(\tan\widetilde\theta)\leqslant\ell(\varphi)\leqslant\mathcal{L}^+_g(f).$$ 
This contradiction yields $\widetilde E\equiv\widetilde E_1$, i.e., $B_1$ is a vertex of $\widetilde E_1$. Similarly we can show that $B_2$ is a vertex of $\widetilde E_2$ and hence Item~(i) follows. 

(ii)--(iii) These can be proved by using completely the same argument as in Claims~\ref{Claim21} and~\ref{Claim23}.

(iv) It follows from Items~(ii) and~(iii) that 
$$\ord f(\widetilde\phi(y),y)= x_1\tan\widetilde\theta_1+y_1\quad\text{ and }\quad\ord g(\widetilde\phi(y),y)= x_2\tan\widetilde\theta_2+y_2=x_2\tan\widetilde \theta_1+y_1,$$
i.e.,
$$\ell(\widetilde\phi)=\mu(\tan\widetilde\theta_1)\geqslant \mu(\tan\theta_1)=\ell(\phi).$$
This implies (iv) and hence the claim follows.
\end{proof}
We are now in position to complete the theorem. Applying Claim \ref{Claim24} (possibly infinitely) many times, we obtain $\phi_\infty$ as a final result of sliding of $\phi$ along $f$. This implies that, $x=\phi_\infty(y)$ is a common Newton--Puiseux root of $f$ and $g$ of multiplicities $x_1$ and $x_2$ respectively. Moreover, from the proof Claim~\ref{Claim24}, $x_1y_2\geqslant x_2y_1$, it follows that
$$\ell(\phi)=\dfrac{x_1\tan\theta_1+y_1}{x_2\tan\theta_1+y_2}\leqslant \frac{x_1}{x_2}\leqslant \mathcal L^+_g(f).$$
This contradicts the assumption $\ell(\phi)>\mathcal L^+_g(f)$.
The theorem is proved.
\end{proof}

%%%%%%%%%%%%%%%%%%%%%%%%%%%%%%%%%%%%%%%%%%%%%%%%%%%%%%%%%%%%%%%%%%%%%%%%%%%%%%

%%%%%%%%%%%%%%%%%%%%%%%%%%%%%%%%%%%%%%%%%%%%%%%%%%%%%%%%%%%%%%%%%%%%%%%%%%%%%%
%%%%%%%%%%%%%%%%%%%%%%%%%%%%%%%%%%%%%%%%%%%%%%%%%%%%%%%%%%%%%%%%%%%%%%%%%%%%%%
%%%%%%%%%%%%%%%%%%%%%%%%%%%%%%%%%%%%%%%%%%%%%%%%%%%%%%%%%%%%%%%%%%%%%%%%%%%%%%
\subsection{Second formula for the \L ojasiewicz exponents}

Recall that $\mathcal V_{\mathbb R}(f)$ is the set of real approximations of non-real Newton--Puiseux roots of $f$ as defined in Definition \ref{def21}. Let $\beta_j,\ j=1,\ldots,k,$ be the common real Newton--Puiseux roots of $f$ and $g$ of multiplicity $m_j$ and $n_j$ respectively. 
\begin{theorem}\label{thm32}
Define
$$\mathscr{L}^+_g(f):=\max\left\{\ell(\gamma),\frac{m_j}{n_j}\ \Big|\ \gamma\in \mathcal V_{\mathbb R}(f),\ j=1,\ldots,k \right\}\ \text{ and }\ \mathscr{L}^-_g(f)=\mathscr{L}^+_{\bar g}(\bar f),$$
where $\bar f(x,y):=f(x,-y)$ and $\bar g(x,y):=g(x,-y)$.  
Then the \L ojasiewicz  exponent of $g$ w.r.t $f$ is given by
$$\mathscr{L}_g(f) = \max\left\{\mathscr{L}^+_g(f), \mathscr{L}^-_g(f) \right\}.$$
\end{theorem}
\begin{proof}
Since $\mathcal V_{\mathbb R}(f)\subset\mathcal V_a(f g)$ and $\mathcal V_{\mathbb R}(\bar f)\subset\mathcal V_a(\bar f \bar g)$,
by Theorem \ref{thm31},
$$\mathscr{L}_g(f) =\max\left\{\mathcal{L}^+_g(f), \mathcal{L}^-_g(f)\right\}\geqslant\max\left\{\mathscr{L}^+_g(f), \mathscr{L}^-_g(f)\right\}.$$
%Without loss of generality, suppose that $\mathscr{L}_g(f) =\mathcal{L}^+_g(f).$
Arguing by contradiction, we assume that
 $$\mathscr{L}_g(f) >\max\left\{\mathscr{L}^+_g(f), \mathscr{L}^-_g(f)\right\}.$$
It follows from Theorem \ref{thm31} that there is an analytic arc $\phi$ passing through the origin and not lying in the $x$-axis such that $g\circ \phi\not\equiv 0$ and 
$$\mathscr{L}_g(f) = \ell(\phi)>\max\left\{\mathscr{L}^+_g(f), \mathscr{L}^-_g(f)\right\}.$$
Note that $\phi$ can be parametrized by either 
$$(x = \phi(t), y = t)\quad \text{ or }\quad (x = \phi(t), y = -t),$$ 
where $\phi(t)$ is an element in $\mathbb{R}\{t^{1/N}\}$ for some positive integer number $N$ with $\phi(0)=0$. 
Let $\mathcal E_{\phi}$ be the polynomial associated to the highest Newton edge of $\mathbb P(f,\phi)$.
With no loss of generality, we can assume that $\phi$ has the following property:

{\it For any analytic arc $\widetilde\phi$ passing through the origin not lying in the $x$-axis and having the parametrization
$(x = \widetilde\phi(t), y = t)$
 such that $g\circ \widetilde\phi\not\equiv 0$ and 
$$\mathscr{L}_g(f) = \ell(\widetilde\phi)>\max\left\{\mathscr{L}^+_g(f), \mathscr{L}^-_g(f)\right\},$$
if $\mathcal E_{\widetilde\phi}$ is the polynomial associated to the highest Newton edge of $\mathbb P(f,\widetilde\phi)$, then $\deg\mathcal E_{\widetilde\phi}\geqslant \deg\mathcal E_{\phi}$.
}

Indeed, if there is an analytic arc $\widetilde\phi$ such that this property does not hold, i.e., $\deg\mathcal E_{\widetilde\phi}< \deg\mathcal E_{\phi}$, then it is enough to replace $\phi$ by $\widetilde\phi$ and repeat the process until the property is satisfied.

%We denote by $\bar\Gamma_+, \bar\Gamma_-$ the set of all curves in $\bar\Gamma$ being of the form $(x(t), t)$ and $(x = x(t), y = -t)$ respectively. 
%Without loss of generality we may assume that $\Gamma_+ \neq \emptyset$. 
%Let $\gamma\in\bar\Gamma_+$ such that $\deg\mathcal E_{\gamma}\leqslant \deg\mathcal E_{\phi}$ for all $\phi\in\bar\Gamma_+$, here for each curve $\phi\in\bar\Gamma_+$, $\mathcal E_{\phi}$ denotes the polynomial associated to the highest Newton edge of $\mathbb P(g,\phi)$. 
Let $E_1$ and $E_2$ be the highest Newton edges of $\mathbb P(f,\phi)$ and $\mathbb P(g,\phi)$ respectively. 
For each $i=1,2$, let $\mathcal{E}_i$ and $\theta_i$ be the associated polynomial and the Newton angle of $E_i$ respectively. 
Let $B_i=(x_i,y_i)$ be the vertex of $E_i$ which is not contained in the $y$-axis. %We first see that, $\tan\theta_1\geqslant\tan \theta_2$. 
Then the following statement holds.
%\begin{claim}\label{Claim25} 
%We have $\tan \theta_1=\tan \theta_2$.
%\end{claim}
 \begin{claim}\label{Claim26} 
We have $\tan \theta_1=\tan \theta_2$.
\end{claim}
\begin{proof}
Applying the same argument as in the proof of Claim \ref{Claim21}, we get $\tan \theta_1\leqslant\tan \theta_2.$
%{( The role of $E_1$ and $E_2$ has been changed here? comparing to Claim \ref{Claim21})}
Assume for contradiction that  $\tan \theta_1<\tan \theta_2$. Let 
$$\psi(y)=\phi(y)+c y^{\tan\theta_1}$$
 with a generic number $c$. It follows from Lemma \ref{Lemma22} that
$$\ord f(\psi(y),y)= x_1\tan\theta_1+y_1= \ord f(\phi(y),y)$$
and
 $$\ord g(\psi(y),y)\leqslant x_2\tan\theta_1+y_2<x_2\tan\theta_2+y_2= \ord g(\phi(y),y).$$
These imply
$$\ell(\psi)> \ell(\phi)=\mathscr{L}_g(f),$$
which is a contradiction. 
\end{proof}
\begin{claim}\label{Claim27} 
The polynomial $\mathcal{E}_1$ has only real roots.
\end{claim}
\begin{proof}
Assume for contradiction that $a\not\in \mathbb R$ is a root of $\mathcal{E}_1$. 
It follows from Lemma \ref{Lemma23}(iii) that there exists $\psi\in\mathcal V_{\mathbb R}(f)$ of the form
$$\psi=\phi+ cy^{\tan \theta_1}+\text{higher order terms}$$
with a generic real number $c$.  Applying Lemma \ref{Lemma22}(i) we obtain
$$\ord f(\psi(y),y)=\ord f(\phi(y),y)\quad\text{ and }\quad\ord g(\psi(y),y)=\ord g(\phi(y),y).$$
Therefore $$\ell(\psi)=\ell(\phi)>\mathscr{L}^+_g(f),$$
which contradicts the definition of $\mathscr{L}^+_g(f)$.
\end{proof}
%%\begin{claim}\label{Claim28} 
%%We have $x_1y_2\geqslant x_2y_1$.
%%\end{claim}
%\begin{proof}
%Indeed, if this is not the case, i.e., $x_1y_2< x_2y_1$. Then the function
%$$\mu(t):=\frac{tx_1+y_1}{tx_2+y_2}$$
%is strictly decreasing. Let $\psi(y)=\phi(y)+y^t$ for some $t<\tan\theta_1\leqslant \tan\theta_2$ closed to $\theta_1$. Then
%$$\ell(\psi)=\mu(t)>\mu(\tan\theta_1)\geqslant \ell(\phi)=\mathscr{L}_g(f),$$
%a contradiction. 
%\end{proof}
\begin{claim}\label{Claim29} 
We have
\begin{enumerate}[{\rm (i)}]
\item $x_1y_2= x_2y_1$, and therefore $\ell(\phi)=\dfrac{x_1}{x_2}=\dfrac{t x_1+y_1}{t x_2+y_2}$ for all $t.$
\item The polynomial $\mathcal{E}_1 \mathcal{E}_2$ has only one root.
\end{enumerate}
\end{claim}
\begin{proof}
(i) First of all, let us prove 
\begin{equation}\label{1221}
x_1y_2\geqslant x_2y_1.
\end{equation}
Indeed, if this is not the case, i.e., $x_1y_2< x_2y_1$, then the function
$$\mu(t):=\frac{tx_1+y_1}{tx_2+y_2}$$
is strictly decreasing. Let $\psi(y)=\phi(y)+y^\rho$ with $0<\rho<\tan\theta_1$ closed enough to $\theta_1(=\theta_2$ by Claim~\ref{Claim26}) so that $\arctan \rho$ is larger than the other Newton angles of $\mathbb P(f,\phi)$ and $\mathbb P(g,\phi)$. Then
$$\ord f(\psi(y),y)= x_1\rho+y_1\quad \text{ and }\quad \ord g(\psi(y),y)= x_2\rho+y_2.$$
So we get
$$\ell(\psi)=\mu(\rho)>\mu \tan\theta_1=  \ell(\phi)=\mathscr{L}_g(f),$$
which contradicts the definition of $\mathscr{L}_g(f)$.
Hence $x_1y_2\geqslant x_2y_1.$ Let us now prove that the equality always holds.

Let $c_j\in\mathbb R, j=1,\ldots,q,$ be the roots of $\mathcal{E}_1(z)$ of multiplicity $x^j_1$ with $x^j_1>0$ and $q\geqslant 1$. We write
$$\mathcal{E}_2(z)=a(z)\prod^q_{j=1}(z-c_j)^{x^j_2}$$
with $x^j_2\geqslant 0$ and $a(c_j)\neq 0$. Observe that 
\begin{equation}\label{deg}
\sum^q_{j=1}x^j_1=\deg \mathcal{E}_1=x_1\quad\text{ and }\quad \sum^q_{j=1}x^j_2+\deg a(z) =\deg \mathcal{E}_2=x_2.
\end{equation}
Let us denote by $A^j_i=(x^j_i,y^j_i)$ the intersection of the line $\{x=x^j_i\}$ with the edge $E_i$ for each $i=1,2$. 
Set $$\mu_{j}(t):=\frac{tx^j_1+y^j_1}{tx^j_2+y^j_2}.$$ 
Since $A^j_i\in E_i$, it follows that
\begin{align}\label{equa22}
\mu_{j} \tan\theta_1=\ell(\phi)\quad\text{ and }\quad y^j_i=y_i+(x_i-x^j_i)\tan\theta_1,
\end{align}
for all $ i=1,2,\ j=1\ldots,q$. We also notice that $A^j_1$ is a vertex of the Newton
polygon $\mathbb P(f,\widetilde\phi_j)$ with
$\widetilde{\phi}_j(y):=\phi(y)+c_j y^{\theta_1}$. 
We shall show that 
\begin{equation}\label{IE}
x^j_1 y^j_2\leqslant x^j_2y^j_1\quad \text{for all } j=1,\ldots,q.
\end{equation}
In fact, by contradiction, assume that $x^j_1 y^j_2>x^j_2y^j_1$, i.e., the function $\mu_{j}(t)$ is strictly increasing. 
From this and~\eqref{equa22}, for  $\rho>\tan\theta_1$ sufficiently closed to $\tan\theta_1$, we have
$$\mathscr{L}_g(f)=\ell(\phi)=\mu_{j} \tan\theta_1 <\mu_{j}\rho=\ell\left(\widetilde{\phi}_j(y)+cy^\rho\right)\leqslant\mathscr{L}_g(f)$$
  for every non-zero $c\in\mathbb R$, which is clearly a contradiction. Thus~\eqref{IE} must holds.
Combining~\eqref{equa22} and~\eqref{IE} yields
$$x^j_1 [y_2+\tan\theta_1(x_2-x^j_2)]\leqslant x^j_2[y_1+\tan\theta_1(x_1-x^j_1)]\quad \text{for all } j=1,\ldots,q.$$
Summing up we obtain
$$(y_2+x_2\tan\theta_1)\sum^q_{j=1}x^j_1\leqslant (y_1+x_1\tan\theta_1)\sum^q_{j=1}x^j_2.$$
Combining this with~\eqref{deg}, we get
$$(y_2+x_2\tan\theta_1)x_1\leqslant(y_1+x_1\tan\theta_1)(x_2-\deg a(z))\leqslant(y_1+x_1\tan\theta_1)x_2.$$
Equivalently
$$x_1y_2\leqslant x_2y_1.$$
By this and~\eqref{1221}, we have $x_1y_2=x_2y_1$ and Item~(i) follows.

(ii) By Item~(i), it follows that $x^j_1 y^j_2=x^j_2y^j_1$ for all $j=1,\dots,q$ and $\deg a(z)=0$.
Hence the function $\mu_{j}(t)$ is constant.
Consider, for each $j$, the curve $\psi_j(y)=\widetilde \phi_j+y^\rho$ for some $\rho>\tan\theta_1$ sufficiently closed to $\tan\theta_1$. Then
$$\ell(\psi_j)=\mu_{j}(\rho)=\mu_{j}\tan\theta_1=\ell(\phi).$$
Moreover, it is not hard to check that $A^j_1$ is a vertex of the Newton polygon $\mathbb P(f,\psi_j)$.
So $\deg \mathcal E_{\psi_j}=x_1^j$ where $\mathcal E_{\psi_j}$ is the polynomial associated to the highest Newton edge of $\mathbb P(f,\psi_j)$.
Then it follows from the choice of $\phi$ that $x^j_1\geqslant x_1$. 
This implies $q=1$ and therefore, by the fact that $\deg a(z)=0$, the polynomial $\mathcal E_1 \mathcal E_2$ must have only one root. The claim is proved.
\end{proof}
Let $a\in\mathbb R$ be the unique root of the polynomial $\mathcal{E}_1 \mathcal{E}_2$ and let $\widetilde\phi(y):=\phi(y)+ a y^{\tan\theta_1}$. Let $\widetilde{\mathbb P}_1:=\mathbb P(f,\widetilde\phi)$ and $\widetilde{\mathbb P}_2:=\mathbb P(g,\widetilde\phi)$. We denote by $\widetilde E_i$ the Newton edge of $\widetilde{\mathbb P}_i$ containing $B_i$ as the vertex with the larger $x$-coordinate. Let $\widetilde\theta_i$ and $\widetilde{\mathcal{E}}_i$ be the Newton angle and the polynomial associated to $\widetilde E_i$ respectively. % Let $B_i(x_i,y_i)$ be the vertex of $E_i$ which is not contained in the $y$-axis $Oy$. Note that $x_i=\deg \mathcal{E}_i$.
%\begin{claim}\label{Claim31}  We have $\tan \widetilde\theta_1\leqslant\tan \widetilde\theta_2$.
%\end{claim}
%\begin{proof}
%%Arguing by contradiction we suppose that $c\in \mathbb{R}$ is a root of $\mathcal{E}_1(z)$.
%We assume by contradiction that, $\tan \widetilde\theta_1>\tan \widetilde\theta_2$. Consider the curve $x=\psi(y)=\widetilde\phi+cy^{\tan\widetilde\theta_1}$ with a generic number $c$. It then follows from               Lemma \ref{Lemma22} that, for any $B_2\neq (a,b)\in \widetilde E_2$,
%$$\ord f(\psi(y),y)\leqslant a\tan\widetilde\theta_1+b<x_2\tan\widetilde\theta_1+y_2\text{ and }\ord g(\psi(y),y)=x_1\tan\widetilde\theta_1+y_1.$$
%Therefore $\ell(\psi)>\ell(\phi)=\mathscr{L}_g(f)$, a contradiction. 
%\end{proof}
\begin{claim}\label{Claim32}  We have $\tan \widetilde\theta_1=\tan \widetilde\theta_2$.
\end{claim}
\begin{proof}
%First of all, let us show that 
%$$\tan \widetilde\theta_1\leqslant\tan \widetilde\theta_2.$$
Assume for contradiction that $\tan \widetilde\theta_1>\tan \widetilde\theta_2$. Consider the curve 
$$\psi(y)=\widetilde\phi+cy^{\tan\widetilde\theta_1}$$ 
with a generic number $c$. 
Then it follows from Lemma~\ref{Lemma22}(i) that, for any $(u,v)\in \widetilde E_2$ such that $(u,v)\ne B_2$, we have
$$\ord f(\psi(y),y)=x_1\tan\widetilde\theta_1+y_1$$
and
$$\ord g(\psi(y),y)\leqslant u\tan\widetilde\theta_1+v<x_2\tan\widetilde\theta_1+y_2.$$
Therefore 
$$\ell(\psi)>\dfrac{x_1\tan\widetilde\theta_1+y_1}{x_2\tan\widetilde\theta_1+y_2}=\dfrac{x_1\tan\theta_1+y_1}{x_2\tan\theta_1+y_2}=\ell(\phi)=\mathscr{L}_g(f),$$
where the first equality follows from~Claim~\ref{Claim29}(i).
This is a contradiction.

Now, by contradiction, suppose that $\tan \widetilde\theta_1<\tan \widetilde\theta_2$. 
%Let us show that $\widetilde\theta_1$ is not a Newton angle of $\widetilde{\mathbb P}_2$. 
%Indeed, assume for contradiction that $\widetilde\theta_1$ is not a Newton angle associated to {a} Newton edge $\widetilde E$ of $\widetilde{\mathbb P}_2$.
%Let $\widetilde B=(a,b)$ be the point of $\widetilde E$ with the largest $x$-coordinate. Then the function
%$$\mu(t):=\frac{tx_1+y_1}{ta+b}$$
%is strictly decreasing and $\mu(\tan \widetilde\theta_1)=\ell(\phi)=\mathscr{L}_g(f)$. 
%Let $\psi(y)=\phi(y)+y^\rho$ for some $\rho<\tan\theta_1$ sufficiently closed to $\tan\theta_1$. 
%Then by Lemma \ref{Lemma22}
%$$\ell(\psi)=\mu(\rho)>\mu(\tan\theta_1)= \ell(\phi)=\mathscr{L}_g(f),$$
%which is a contradiction. Therefore $\widetilde\theta_1$ is not a Newton angle of $\widetilde{\mathbb P}_2$. 
Let us show that the polynomial  $\widetilde{\mathcal{E}}_1$ has only real root. In fact, if this is not the case,  then by Lemma \ref{Lemma23}(iii), there exists $\psi\in \mathcal V_\mathbb R(f)$ of the form
$$\psi(y)=\widetilde\phi+cy^{\tan\widetilde\theta_1}$$ with a generic number $c$. It then follows from Lemma \ref{Lemma22}(i) that 
$$\ord f(\psi(y),y)= x_1\tan\widetilde\theta_1+y_1\quad\text{ and }\quad\ord g(\psi(y),y)\leqslant x_2\tan\widetilde\theta_1+y_2.$$
Therefore, in view of Claim~\ref{Claim29}(i), 
$$\ell(\psi)\geqslant\dfrac{x_1\tan\widetilde\theta_1+y_1}{x_2\tan\widetilde\theta_1+y_2}=\dfrac{x_1\tan\theta_1+y_1}{x_2\tan\theta_1+y_2}= \ell(\phi)=\mathscr{L}_g(f),$$ 
which is a contradiction, because $\psi\in\mathcal V_\mathbb R(f)$. 

We now take $0\neq a\in \mathbb R$ such that $\widetilde{\mathcal{E}}_1(a)=0$ and define $\gamma(y)=\widetilde\phi+ay^{\tan\widetilde\theta_1}$. Then it follows from Lemma \ref{Lemma22}(i) that
$$\ord f(\gamma(y),y)>x_1\tan\widetilde\theta_1+y_1\quad\text{ and }\quad\ord g(\gamma(y),y)\leqslant x_2\tan\widetilde\theta_1+y_2.$$
Therefore 
$$\ell(\gamma)>\dfrac{x_1\tan\widetilde\theta_1+y_1}{x_2\tan\widetilde\theta_1+y_2}=\ell(\phi)=\mathscr{L}_g(f),$$ 
a contradiction. Hence $\tan \widetilde\theta_1=\tan \widetilde\theta_2$.
\end{proof}
\begin{claim}\label{Claim33}  If $\widetilde\phi(y)$ is not a Newton--Puiseux root of $f$, then it and the Newton polygons of $f$ and $g$ relative to it share the following properties with that of $\phi$:
\begin{itemize}
\item[(i)] $\tan \widetilde\theta_1=\tan \widetilde\theta_2$.
\item[(ii)] The polynomial $\widetilde{\mathcal{E}}_1 \widetilde{\mathcal{E}}_2$ has only one root. In particular, for each $i=1,2$, $\widetilde E_i$ is the highest Newton edge of $\widetilde{\mathbb P}_i$.
\item[(iii)]$\ell(\widetilde\phi)=\dfrac{x_1}{x_2}=\ell(\phi)$.
\end{itemize}
\end{claim}
\begin{proof}
%Arguing by contradiction we suppose that $c\in \mathbb{R}$ is a root of $\mathcal{E}_1(z)$.
It is clear that Item~(i) follows from Claim \ref{Claim32}.
Furthermore, Items~(ii) and~(iii) can be proved by using the same argument as in the proof of Claim \ref{Claim29}.
\end{proof}
We are now in position to complete the theorem. Applying Claim
\ref{Claim33} (possibly infinitely) many times, we obtain a final
result $\phi_\infty$ of sliding of $\phi$ along $f$ which
is also that of $g$. This implies that, $\phi_\infty$ is
a common Newton--Puiseux root of $f$ and $g$ of multiplicities
$x_1$ and $x_2$ respectively. Therefore,
$$\ell(\phi)=\frac{x_1}{x_2}\leqslant \mathscr L^+_g(f).$$
This contradicts the assumption that $\ell(\phi)>\mathscr L^+_g(f)$. Hence the theorem follows.
\end{proof}
\section{Algorithms}
In this section we provide an algorithm verifying whether $\{f=0\}\subset \{g=0\}$ and computing the Lojasiewicz {exponent} $\mathscr L_g(f)$ if it is defined. Let $f\in \Bbb R[x,y]$ be regular in $x$ and let $\mathcal V(f)$ be the set of all {Newton--Puiseux} roots $x=\gamma(y)$ of $f$.  Let $x=\varphi(y)$ be a (complex) Puiseux series. The {\em contact order} of $\varphi$ and $f$ is defined as
$$\rho(\varphi,f):=\max\{\mathrm{ord}\left(\varphi(y)-\gamma(y)\right)\mid \varphi\neq \gamma\in \mathcal V(f)\}.$$
For each rational number $q$, the series $\varphi$ is called a {Newton--Puiseux} root$\mod q+$ of $f$ if there exists $\gamma\in \mathcal V(f)$ such that $\mathrm{ord}\left(\varphi(y)-\gamma(y)\right)>q$.
Assume that
 $$x=\gamma(y)=\sum c_\alpha y^\alpha$$
is a {Newton--Puiseux} root of $f$, then the series
$$\tilde\gamma(y)=\sum_{\alpha\leq \rho} c_\alpha y^\alpha,$$
where $\rho=\rho(\gamma,f)$, is called {\em a truncated {Newton--Puiseux} root} of $f$. We denote by $\tilde{\mathcal V}(f)$ the set of  truncated {Newton--Puiseux} roots $f$. 
\begin{remark}\label{rm41}
{It follows from the definition that:}	
\begin{itemize}
\item[(i)] If $\gamma(y)$ and $\gamma'(y)$ are distinct {Newton--Puiseux} roots of $f$, then $\tilde\gamma\neq \tilde\gamma'$. That is, the natural map ${\mathcal V}(f)\to \tilde{\mathcal V}(f)$ is bijective.
\item[(ii)] If $\gamma(y)$ is a {Newton--Puiseux} root of $f$ then
$$\mathrm{ord}\left(\tilde\gamma(y)-\gamma(y)\right)>\rho(\gamma,f):=\max\{\mathrm{ord}\left({\gamma}(y)-\gamma'(y)\right)\mid \gamma\neq \gamma'\in \mathcal V(f)\}.$$
\end{itemize}
\end{remark}
\begin{theorem}
Let $f\in \Bbb R[x,y]$ be regular in $x$ and let $x=\gamma(y)$ be a {Newton--Puiseux} root of $f$. Let {$\rho=\rho(\gamma,f)$} the contact order of $\gamma$ and $f$. Then
\begin{itemize}
\item[(i)] If the truncated {Newton--Puiseux} root $\tilde\gamma$ of $\gamma$ is real, then $\gamma$ is a real {Newton--Puiseux}  root of $f$.
\item[(ii)] We write
$$f(X + \tilde\gamma(y), Y) = \sum c_{ij}X^iY^{j/N}.$$
Then {the multiplicity of $\gamma$, denoted by} $\mathrm{mult}_{\gamma}f$, is equal to the minimum of $i$ such that
\begin{align}\label{mult}
i\rho+j/N=\mathrm{ord}(f(\tilde\gamma_\rho(y),y))\text{ and } c_{ij}\neq 0,
\end{align}
where $\tilde\gamma_\rho$ is the $\rho$-approximation of $\tilde\gamma$.
\item[(iii)] Let $g\in \Bbb R[x,y]$ be regular in $x$ and let $h:=\gcd(f,g)$. If $\tilde\gamma$ is a root $\mod {\rho(\gamma,f)+}$ of $h$ then $\gamma(y)$ is also a root of $g$.
\end{itemize}
\end{theorem}
\begin{proof}
{(i) Assume for contradiction that $\gamma$ is not real and write
$$\gamma(y)=\phi(y)+cy^\alpha + \textrm{ higher order terms},$$
where $\phi$ is the sum of terms of order lower than $\alpha$ with real coefficients and $c\in \mathbb{C}\setminus \mathbb{R}$.
Since $f$ is real, the conjugate
$$\bar{\gamma}(y)=\phi(y)+\bar{c}y^\alpha + \textrm{ higher order terms}$$
is also a Newton--Puiseux root of $f.$
Thus
$$\begin{array}{lll}
\rho(\gamma,f)&=&\max\{\mathrm{ord}\left(\gamma(y)-\gamma'(y)\right)\mid \gamma\neq \gamma'\in \mathcal V(f)\}\\
&\geq&\mathrm{ord}\left(\gamma(y)-\bar\gamma(y)\right)=\alpha.
\end{array}$$
By definition of truncated Newton--Puiseux root, $\tilde\gamma$ contains the term $cy^\alpha$ so it is not real which is a contradiction.
Consequently $\gamma$ is real.}

Let $\Bbb P(f,\tilde\gamma)$ be the Newton polygon of $f$ relative to $\tilde\gamma$ and let $E_1,\ldots,E_s$ be its Newton edges. Let $\theta_i$ and $\mathcal E_i$ be the Newton angle and the polynomial associated to $E_i$ respectively. Consider a progress of recursive slidings 
$$\tilde\gamma\to\tilde\gamma_1\to\ldots\to\tilde\gamma_\infty$$
of $\tilde\gamma$ along $f$. The following claim is a direct consequence of Lemma \ref{Lemma22}. 
\begin{claim}\label{Claim41}  
We have
 $$\mathrm{ord}\left(\tilde\gamma(y)-\tilde\gamma_\infty(y)\right)=\max \{\mathrm{ord}\left(\tilde\gamma(y)-\gamma'(y)\right)\mid \gamma'\in \mathcal V(f)\}=\theta_1.$$
and   $$\rho=\max \{\mathrm{ord}\left(\tilde\gamma(y)-\gamma'(y)\right)\mid \gamma\neq\gamma'\in \mathcal V(f)\}=\theta_2.$$
\end{claim}
This together with Remark \ref{rm41} implies that $\tilde\gamma_\infty=\gamma$. This means that, there is a unique progress of  recursive slidings of $\tilde\gamma$ along $f$. Write
$$\tilde\gamma_\infty=\tilde\gamma+a_1 y^{\alpha_1}+a_2 y^{\alpha_2}+\ldots.$$
Then for all $n\geq 1$,
$$\tilde\gamma_n=\tilde\gamma+a_1 y^{\alpha_1}+a_2 y^{\alpha_2}+\ldots+a_n y^{\alpha_n}.$$

Since $\tilde\gamma_n$ is the only sliding of $\tilde\gamma_{n-1}$ along $f$, the polynomial $\mathcal{E}_H^{n-1}$ of associated to the highest Newton edges of $\mathbb P(f,\tilde\gamma_{n-1})$ has only one root $a_n$ of multiplicity $\deg \mathcal{E}_H^{n-1}=\deg \mathcal{E}_H^{0}=\deg \mathcal{E}_1$. Then the multiplicity $\mathrm{mult}_{\gamma}f$ of $\gamma$ is equal to 
$$\deg \mathcal{E}_1=\mathrm{ord}\ \mathcal{E}_2=\min\{i\mid (i,j/N)\in E_2\}.$$
 Since 
$$E_2=\{(i,j/N)\in \mathrm{supp}(f)\mid i\theta_2+j/N=\mathrm{ord}(f(\tilde\gamma_\rho(y),y))\}$$
it follows that 
$$\mathrm{mult}_{\gamma}f=\min\{i\mid i\rho+j/N=\mathrm{ord}(f(\tilde\gamma_\rho(y),y))\text{ and } c_{ij}\neq 0\},$$
which gives (ii).
Now, we take a root $\xi$ of $h$ such that $\mathrm{ord}\left(\xi(y)-\tilde\gamma(y)\right)>\rho(\gamma,f)$. Then, it follows from the definition of $\rho(\gamma,f)$ that $\xi=\gamma$. This completes (iii).
\end{proof}
{ As a consequence,} we obtain the following algorithm for computing the \L ojasiewicz exponent $\mathcal{L}_g(f)$.
\subsection*{Algorithm BiLojEx} \

INPUT: Two polynomials $f$ and $g$ in $\mathbb{Q}[x,y]$ of positive orders. % $m$ and $n,$ respectively.

OUTPUT: Decide whether or not $\{f=0\}\subset \{g=0\}$ and compute the \L ojasiewicz exponent $\mathcal{L}_g(f)$.

\begin{enumerate}[{\rm Step 1.}]

%\item If $g(0, 0) \ne 0,$ the limit $\displaystyle\lim_{(x, y) \to (0, 0)}\frac{f(x, y)}{g(x, y)}$ exists and equals to $\displaystyle\frac{f(0, 0)}{g(0, 0)};$ the algorithm stops.

%\item If $m < n,$ the limit $\displaystyle\lim_{(x, y) \to (0, 0)}\frac{f(x, y)}{g(x, y)}$ does not exist and go to the next step.
%If $m > n,$ let $L := 0$ and go to the next step. Otherwise (i.e., $m = n$), take any $(\bar{x}, \bar{y}) \in \mathbb{R}^2 \setminus (f_m^{-1}(0) \cup 
%g_n^{-1}(0))$ and let $L := \frac{f_m(\bar{x}, \bar{y})}{g_n(\bar{x}, \bar{y})}.$ If $f_m \not \equiv L g_n,$ the limit $\displaystyle\lim_{(x, y) \to (0, 0)}\frac{f(x, y)}{g(x, y)}$ does not exist and go to next step. 
%Otherwise, replace $f$ by $f - L g$ (then $\ord f > \ord g$) and proceed to the next step.

\item If one of the polynomials $f, g$ is not $x$-regular, make a linear transformation, so that the new polynomials $f, g$ are $x$-regular. Compute $h:=\gcd(f,g)$.

%\item If $\left(g_n\frac{\partial f_m}{\partial x}-f_m\frac{g_n}{\partial x}\right)(0,1)  = 0,$ replace $f$ and $g$ by $f(x, x + y)$ and $g(x, x + y),$ respectively, and compute the new polynomials $F_{f,g}$ and $G_{f,g}.$ (Then, all $f, g, F_{f,g}$ and $G_{f,g}$ are $y$-regular.)

\item Compute the set $\tilde{\mathcal{V}}(f)$  of truncated roots of $f$. Compute the sets $\tilde{\mathcal{V}}_\Bbb R(f)$ and $\tilde{\mathcal{V}}_\Bbb R(h)$ of truncated real roots of $f$ and $h$.
  
If $\sharp\ \tilde{\mathcal{V}}_\Bbb R(f)\leq \sharp\ \tilde{\mathcal{V}}_\Bbb R(h)$ then $\{f=0\}\subset \{g=0\}$ and proceed to the next step. Otherwise, the \L ojasiewicz exponent $\mathcal{L}_g(f)$ is not defined and the algorithm stops.

\item Compute for each ${\gamma}\in \tilde{\mathcal{V}}_\Bbb R(f)$ the multiplicies $\mathrm{mult}_{\gamma} f$ and $\mathrm{mult}_{\gamma} g$ by Formula (\ref{mult}).

\item Compute the set $\tilde{\mathcal{V}}_a(f)$ of the real approximations of series in $\tilde{\mathcal{V}}(f)\setminus \tilde{\mathcal{V}}_\Bbb R(f)$ and compute
$$\mathscr{L}^+_g(f):=\max\left\{\ell(\gamma),\frac{\mathrm{mult}_{\gamma} f}{\mathrm{mult}_{\gamma} g}\ \Big|\ \gamma\in\tilde{\mathcal{V}}_a(f),{\gamma}\in  \tilde{\mathcal{V}}_\Bbb R(f)\right\}.$$

\item Set ${\tilde f}(x,y):=f(x,-y)$ and ${\tilde g}(x,y):=g(x,-y)$ and compute $\mathscr{L}^+_{ \tilde g}({\tilde f})$.

\item $\mathscr{L}_g(f):=\max\{\mathscr{L}^-_g(f),\mathscr{L}^+_{ \tilde g}({\tilde f})\}$.
\end{enumerate}

\section{Applications}
Computing limits of (real) multivariate functions at given points is one of the basic problems in computational mathematics. Let $\frac{g}{f}$ be a rational function with $f,g$ real polynomials. It is well known that if the limit $\lim_{(x,y)\to (0,0)} \frac{g(x,y)}{f(x,y)}$ exists, it can be easily computed by evaluating the limit along a ray $R$ through $(0,0)$. Therefore, replacing $\lim_{(x,y)\to (0,0)} \frac{g(x,y)}{f(x,y)}$ by $\lim_{(x,y)\to (0,0)} \frac{g(x,y)-L f(x,y)}{f(x,y)}$ with $L=\lim_{R\ni(x,y)\to (0,0)} \frac{g(x,y)}{f(x,y)}$ for some ray $R$, one reduces the problem to studying whether $\lim_{(x,y)\to (0,0)} \frac{g(x,y)}{f(x,y)}=0$.
%\section{Limits of bivariate rational functions}
The following sufficient condition is straightforward.
\begin{proposition}
\begin{enumerate}
\item If  $0<\mathcal{L}_g(f)<1$, then
$$\displaystyle\lim_{(x,y)\to (0,0)} \frac{g(x,y)}{f(x,y)}=0.$$
\item  If  $\mathcal{L}_g(f)>1$, then  the limit $\lim_{(x,y)\to (0,0)} \frac{g(x,y)}{f(x,y)}$ does not exist.
\end{enumerate}
\end{proposition}
In the case when $\mathcal{L}_g(f)=1$ the limit $\lim_{(x,y)\to (0,0)} \frac{g(x,y)}{f(x,y)}$ may exist or not. However, we can deduce the following corollary from the proof of our main results (Theorem \ref{thm32}).
\begin{corollary}\label{limit}
Let $f\in \Bbb R[x,y]$ be regular in $x$ and $\mathcal V_{\mathbb R}(f)$ be the set of real approximations of non-real Newton--Puiseux roots of $f$ as defined in Definition \ref{def21}. Asume that $f$ and $g$ have no common factors, then 
$$\lim_{(x,y)\to (0,0)} \frac{g(x,y)}{f(x,y)}=0$$
if and only if $f=0$ has only isolated point $(0,0)$ and
$$\lim_{y\to 0} \frac{g(\phi(y),y)}{f(\phi(y),y)}=0$$
for all $\phi\in \mathcal V_{\mathbb R}(f)$.
\end{corollary}
This provides a new algorithsm, which are easy to implement, to determine whether the limit $\lim_{(x,y)\to (0,0)} \frac{f(x,y)}{g(x,y)}$ exists and compute the limit if it exists. 
\subsection*{Acknowledgment} 
A part of this work was done while the first author and the third author were visiting at Vietnam Institute for Advanced Study in Mathematics (VIASM) in the spring of 2022. These authors would like to thank the Institute for hospitality and support. 
\subsection*{Conflict of interest} 
On behalf of all authors, the corresponding author states that there is no conflict of interest. 
\subsection*{Data availability} 
Data sharing is not applicable to this article as no new data were created or analyzed in this study. 

\bibliographystyle{abbrv}

\end{document}